\documentclass[11pt,a4paper]{article}
\usepackage{amsmath,amsthm,amssymb}
\usepackage{graphicx}
\usepackage{subfigure}
\usepackage{multirow}
\usepackage{epstopdf}
\usepackage{verbatim}
\usepackage{color}

\usepackage{tikz}
\usetikzlibrary{matrix}

\usepackage{comment}

\newtheorem{theorem}{Theorem}[section]
\newtheorem{lemma}{Lemma}[section]
\newtheorem{proposition}{Proposition}[section]

\newtheorem{remark}{Remark}[section]

\numberwithin{equation}{section}

\newcommand{\cop}{^{C}\!\partial_t^{\alpha}}
\newcommand{\op}{\partial_t^{\alpha}}
\newcommand{\dop}{\partial_\tau^{\alpha}}
\newcommand{\dopm}{\partial_\tau^{-\alpha}}

\newcommand{\cL}{\mathcal{L}}

\newcommand{\bv}{{\bf w}}
\newcommand{\bs}{{\boldsymbol\sigma}}
\newcommand{\bz}{{\boldsymbol\zeta}}
\newcommand{\bx}{{\boldsymbol\xi}}

\newcommand{\bV}{{\bf W}}

\newcommand{\eh}{{\ell_h(\nu)}}

\title{Galerkin Type Methods for Semilinear Time-Fractional Diffusion Problems}

\author{
Samir Karaa\thanks{Email: skaraa@squ.edu.om. This research was supported by the Research Council of Oman grant
RC/SCI/DOMS/16/01.} 
\\
Department of Mathematics, Sultan Qaboos University\\
Al-Khod 123, Muscat, Oman
}

\begin{document}
\date{}
\maketitle

\begin{abstract}
We derive optimal $L^2$-error estimates for semilinear time-fractional subdiffusion problems involving  Caputo derivatives in time of order $\alpha\in (0,1)$, for cases with smooth and nonsmooth initial data. A  general framework is introduced allowing a unified error analysis of Galerkin type space approximation methods.
The analysis is based on a semigroup type approach and exploits the properties of the inverse of the associated elliptic operator.  Completely discrete schemes are analyzed in the same framework using a backward Euler convolution quadrature method in time. Numerical examples including conforming, nonconforming and mixed finite element  (FE) methods are presented to illustrate the theoretical results.
\end{abstract}

{\small{\bf Key words.} semilinear fractional diffusion, Galerkin method, nonconforming FE method, mixed FE method,  convolution quadrature, error estimate


\section{Introduction}

The purpose of this paper is to discuss some aspects of the numerical solution of the semilinear 
time-fractional initial boundary value problem
\begin{equation}\label{a}
\cop u +\cL u=f(x,t,u)\; \mbox{ in } \Omega\times (0,T_0],\quad u(x,0)=u_0(x)\; \mbox{ in }  \Omega,
\end{equation}
subject to a homogeneous Dirichlet boundary condition, where $\Omega\subset \mathbb{R}^d$ ($d\geq 2$) is a bounded convex polyhedral domain with a boundary $\partial \Omega$ and  $T_0>0$ is a fixed time. Here $u_0$ is a given initial data 
and $f$ is a smooth function of its arguments satisfying
\begin{equation} \label{Lip}
\sup_{x\in\Omega,t\in (0,T_0)}\big(|\partial_t f(x,t,u)|+|\partial_u f(x,t,u)|\big)\leq L \quad \forall u\in \mathbb{R}.
\end{equation} 
The operator $\cL$ is defined by $\cL u = -\mbox{div} [A(x)\nabla u]+\kappa(x)u$, where $A(x)=[{a_{ij}(x)}]$ is a  $d\times d$ symmetric and  uniformly positive definite in $\bar\Omega$ matrix, and   $\kappa\in L^\infty(\Omega)$ is nonnegative. The coefficients  $a_{ij}$ and $\kappa$ are assumed to be sufficiently smooth on $\bar{\Omega}$.
The operator $\cop$ is the Caputo fractional derivative in time of order $\alpha\in(0,1)$  defined by
\begin{equation} \label{Ba}
\cop\varphi(t)=\frac{1}{\Gamma(1-\alpha)}\int_0^t(t-s)^{-\alpha}\partial_s\varphi(s)\,ds,\quad 0<\alpha<1,
\end{equation}
where $\partial_s\varphi=\partial \varphi/\partial s$ and $\Gamma(\cdot)$ denotes the usual Gamma function. 
As $\alpha \rightarrow 1^-$,  $\cop$  converges
to $\partial_t$, and thus, problem \eqref{a} reduces to the standard semilinear parabolic problem \cite{KST-2006}. 
 
Let $(\cdot,\cdot)$ denote the inner product in $L^2(\Omega)$ with induced norm $\|\cdot\|$. 
Since $\Omega$ is convex,  the solution of the elliptic problem $\cL u=f$ in $\Omega$, with $u=0$ on $\partial \Omega$ and $f\in L^2(\Omega)$,  belongs to $H^2(\Omega)$. With $\mathcal{D}(\cL)=H^2(\Omega)\cap H^1_0(\Omega)$, recall that the operator  $\cL:\mathcal{D}(\cL)\to L^2(\Omega)$ is selfadjoint, positive definite and has a compact inverse. Let $\{\lambda_j,\varphi_j\}_{j=1}^\infty$ denotes the eigenvalues and  eigenfunctions of $\cL$ with $\{\varphi_j\}_{j=1}^\infty$ an orthonormal basis in $L^2(\Omega)$. By spectral method, the fractional powers of $\cL$ are defined by
$$
\cL^\nu v = \sum_{j=1}^\infty \lambda_j^\nu (v,\varphi_j)\varphi_j, \quad\nu>0,
$$
with domains $\mathcal{D}(\cL^\nu)=\{v\in L^2(\Omega): \|\cL^\nu v\|<\infty\}$. Note that  
$\{\mathcal{D}(\cL^\nu)\}$ form a Hilbert scale of interpolation spaces and  
$\mathcal{D}(\cL)\subset \mathcal{D}(\cL^\nu)\subset\mathcal{D}(\cL^\beta)
\subset\mathcal{D}(\cL^0)= L^2(\Omega)$ with continuous and compact embeddings for $0<\beta<\nu<1$.

The regularity of the solution in \eqref{a} plays a key role in our error analysis.
For initial data $u_0\in \mathcal{D}(\cL^\nu)$, $\nu\in (0,1]$, 
 problem  \eqref{a}  has a unique solution  $u$ satisfying  \cite[Theorem 3.1]{Karaa-2018b}:
\begin{equation}\label{regularity-1}
u\in C^{\alpha \nu}([0,T_0];L^2(\Omega))  \cap  C([0,T_0];\mathcal{D}(\cL^\nu))
\cap   C((0,T_0];\mathcal{D}(\cL)),
\end{equation}
\begin{equation}\label{regularity-1b}
 \cop u\in C((0,T_0];L^2(\Omega)),
\end{equation}
\begin{equation}\label{regularity-1c}
\partial_t u(t)\in L^2(\Omega) \quad \mbox{and}\quad    \|\partial_t u(t)\|\leq c t^{\alpha \nu -1}, \;\; \quad t\in (0,T_0].
\end{equation}
The results show that the solution of the semilinear problem \eqref{a} enjoys  (to some extent)  smoothing properties analogous to those of the homogeneous linear problem. 
For $u_0\in L^2(\Omega)$, it is shown that (\cite[Theorem 3.2]{Karaa-2018b})
\begin{equation}\label{regularity-1d}
u\in C([0,T_0];L^2(\Omega))  \cap  L^\gamma(0,T_0;\mathcal{D}(\cL)), \quad \gamma <1/\alpha.
\end{equation}
Note that the first time derivative of $u$ is not smooth enough in space even in the case of a  smooth initial data. This actually causes a major difficulty in deriving optimal error estimates based on standard techniques, such as the energy method.


The numerical approximation of fractional differential equations  has received considerable attention
over the  last two decades. For linear time-fractional equations,  a vast literature  is now available. See the short list \cite{MT2010,McLeanThomee2010,  JLZ2015, JLZ2016,Karaa-2018,KP-2018} on problems with nonsmooth data 
and \cite{JLZ2019} for a concise overview and recent developments.
In contrast, numerical studies on  nonlinear time-fractional evolution problems are rather limited.
 In \cite{LWZ-2017}, a linearized $L^1$-Galerkin FEM was proposed to solve a  nonlinear time-fractional Schr\"odinger equation. In \cite{LLSWZ-2018},  $L^1$-type schemes have been analyzed for approximating the solution of \eqref{a}.  The error estimates in \cite{LWZ-2017} and \cite{LLSWZ-2018} are derived under high regularity assumptions on the exact solution, so the limited smoothing property of the model \eqref{a} was not taken into consideration.   In \cite{JLZ-2018}, the numerical solution of \eqref{a} was investigated assuming that  the nonlinearity  $f$ is  uniformly Lipschitz  in $u$ and the  initial data $u_0\in \mathcal{D}(\cL)$. Error estimates are established for linearized time-stepping schemes based on the $L^1$-method and a convolution quadrature generated by the backward Euler difference formula. 
In the recent paper \cite{Karaa-2018b}, we derived error estimates for the same problem with initial data $u_0\in \mathcal{D}(\cL^\nu)$, $\nu\in (0,1]$. 
The new  estimates extend known results obtained for the standard semilinear parabolic problem \cite{JLTW-1987}. For other types of time-fractional problems, one may refer to \cite{Lubich-2006} for  fractional diffusion-wave equations
and to \cite{MM-2010} for an integro-differential equation.


In this paper, we approximate the solution of the semilinear problem \eqref{a} by general Galerkin type approximation methods in space  and a convolution quadrature in time. Our aim is to develop a unified error analysis with optimal error estimates with respect to the data regularity. We shall follow a semigroup type approach and make use of the inverse of the associated elliptic operator \cite{BSTW-1977}. 
The current study extends the recent work \cite{Karaa-2018} dealing with the homogeneous linear problem, which relied on the energy technique. 
Our analysis includes conforming, nonconforming and mixed FEMs, and the results are applicable to  nonlinear multi-term diffusion problems. 
It is worth noting that most of our results hold in the limiting case $\alpha=1$, i.e.,
our study also generalizes the work \cite{BSTW-1977}. Particularly interesting are the 
estimates derived for the mixed FEM, which are new and have not been established earlier.

%

The paper is organized as follows. In section \ref{sec:AB}, a general setting of the problem is introduced and preliminary error estimates are derived, which require regularity properties analogous to those of the homogeneous linear problem. 
In section \ref{sec:RA}, an alternative error estimation is proposed without a priori regularity assumptions on the exact solution. Time-stepping schemes based on a backward Euler convolution quadrature method are analyzed in section \ref{sec:FD}. 
Applications are presented in section  \ref{sec:Appli}.
The mixed form of problem \eqref{a} is considered in section \ref{sec:Mixed} and related convergence rates
are obtained. Finally, numerical results are provided to validate the theoretical  findings.

Throughout the paper, we denote by $c$ a constant which may vary at different occurrences, but is  always independent of the mesh size $h$ and the time step size $\tau$. We shall also use the abbreviation $f(u)$ and $f(t)$ for $f(x,t,u)$ and $f(x,t)$, respectively.


\section{General setting and preliminary estimates}\label{sec:AB}
%
Set $T=\cL^{-1}$. Then,  
 $T:L^2(\Omega)\to \mathcal{D}(\cL)$  is compact, selfadjoint and positive definite.
In terms of $T$, we may write \eqref{a} as
\begin{equation}\label{pr-1}
T\cop u +u = Tf(u),\quad t>0,\quad u(0)=u_0.
\end{equation}
For the purpose of approximating the solution of this problem,  let  $V_h\subset L^2(\Omega)$ be a family of finite-dimensional spaces  that depends on $h$, $0<h<1$. We assume that we are given a corresponding family of linear operators $T_h:L^2(\Omega)\to V_h$ which approximate $T$. Then consider the semidiscrete problem: find $u_h(t)\in V_h$ for $t\geq 0$ such that
 \begin{equation}\label{pr-2}
T_h\cop u_h +u_h = T_hf(u_h),\quad t>0 ,\quad  u_h(0)=u_{0h}\in V_h,
\end{equation} 
where $u_{0h}$ is a suitably chosen  approximation of  $u_0$. In our analysis, we shall make the following assumptions:\\
(i) $T_h$ is selfadjoint,  positive semidefinite on $L^2(\Omega)$ and  positive definite on $V_h$.\\
(ii) $T_hP_h=T_h$, where $P_h:L^2(\Omega)\to V_h$ is the orthogonal $L^2$-projection onto $V_h$.\\
(iii) For some constants $\gamma>0$ and $c>0$, there holds 
\begin{equation}\label{ass-1}
 \|T_hf-Tf\|\leq ch^{\gamma}\|f\|\quad \forall f\in L^2(\Omega). 
\end{equation}

Since $T_h^{-1}$ exists on  $V_h$,  \eqref{pr-2} may  be solved uniquely for $t> 0$.
The following diagram displays the different  links between the operators under consideration:

\vspace*{.1cm}

\begin{center}
\begin{tikzpicture}
  \matrix (m) [matrix of math nodes,row sep=4em,column sep=6.5em,minimum width=2em]
  {
      D(\cL) & L^2(\Omega) \\ 
     V_h & V_h \\};
  \path[-stealth]
  
    (m-1-2) edge  node [above] {$T=\cL^{-1}$} (m-1-1)
            edge node [right] {$P_h$} (m-2-2)
            edge node [above] {$T_h$} (m-2-1)
    (m-1-1) edge node [left] {$R_h$} (m-2-1)
    (m-2-2) edge node [above] {$T_h$} (m-2-1);
\end{tikzpicture}
\end{center}
In the diagram,  the operator $R_h:D(\cL)\to V_h$  is defined by $R_h=T_h\cL$. It is  the analogue of the 
Ritz  elliptic projection in the context of  Galerkin FE methods. 
Note that  $R_hT=T_h$, and 
in view of \eqref{ass-1}, $R_h$ satisfies
\begin{equation}\label{ass-2}
\|R_h v-v\|= \|T_h\cL v-T\cL v\|\leq ch^\gamma\|\cL v\|\quad \forall v\in D(\cL).
\end{equation}
Further, by the definition of $P_h$, we see that
$\|P_h v-v\|\leq \|R_h v-v\| \, \forall v \in \mathcal{D}(\cL)$.

Examples of family $\{T_h\}$ with the above properties are exhibited by the standard Galerkin FE and spectral methods in the case $V_h\subset  H_0^1(\Omega)$, and by  other nonconforming Galerkin  methods in the case $V_h\not\subset  H_0^1(\Omega)$.
The  mixed FE method applied to \eqref{a} is a typical example which has the above properties 
and  will be considered in this study. 

By our assumptions on $T_h$, the operator
$(z^{-\alpha} I+T_h)^{-1}:L^2(\Omega)\to L^2(\Omega)$ satisfies 
\begin{equation}\label{res1}
\|(z^{-\alpha} I+T_h)^{-1}\|\leq M |z|^{\alpha} \quad \forall z\in \Sigma_\theta,
\end{equation}
where  $\Sigma_{\theta}$ is the  sector
$
\Sigma_{\theta}=\{z\in \mathbb{C}, \,z\neq 0,\, |\arg z|< \theta\}
$
with $\theta\in (\pi/2,\pi)$ being fixed and $M$ depends on $\theta$. In \eqref{res1}, and in the sequel,  we keep the same notation $\|\cdot\|$ to denote the operator norm from $L^2(\Omega)\to L^2(\Omega)$. Using that
\begin{equation}\label{expand}
(z^{-\alpha} I+T_h)^{-1}T_h=I- z^{-\alpha} (z^{-\alpha} I+T_h)^{-1},
\end{equation}
we obtain
\begin{equation}\label{res2}
\|(z^{-\alpha} I+T_h)^{-1}T_h\|\leq 1+M \quad \forall z\in \Sigma_\theta.
\end{equation}
Note that \eqref{res1} and \eqref{res2} hold for $T$. 
By means of the Laplace transform, the solution of problem \eqref{pr-2} is represented  by
\begin{equation}\label{form-1}
u_h(t)=E_h(t)u_{0h}+\int_0^t {\bar E}_h(t-s)f(u_h(s))\,ds,\quad t>0.
\end{equation}
The operators $E_h(t):L^2(\Omega)\to L^2(\Omega)$ and $\bar E_h(t):L^2(\Omega)\to L^2(\Omega)$ are defined by
\begin{equation}\label{form-1-pl}
 E_h(t) = \frac{1}{2\pi i}\int_{\Gamma_{\theta,\delta}}e^{zt} z^{-1}K_h(z) \,dz
\quad
\mbox{ and }
\quad 
{\bar E}_h(t) = \frac{1}{2\pi i}\int_{\Gamma_{\theta,\delta}}e^{zt}z^{-\alpha} K_h(z)\,dz,
\end{equation}
respectively, where $K_h(z):=(z^{-\alpha} I+T_h)^{-1}T_h$.
 The contour $\Gamma_{\theta,\delta}=\{\rho e^{\pm i\theta}:\rho\geq \delta\}\cup\{\delta e^{i\psi}: |\psi|\leq \theta\},$ 
with $\theta\in (\pi/2,\pi)$  and $\delta> 0$, is 
oriented with an increasing imaginary part. Similarly, the solution $u$ of problem \eqref{pr-1} 
is given  by
\begin{equation}\label{form-2}
u(t)=E(t)u_{0}+\int_0^t {\bar E}(t-s)f(u(s))\,ds,\quad t>0,
\end{equation}
where the operators $E$ and $\bar E$ are respectively defined in terms of $K(z):=(z^{-\alpha} I+T)^{-1}T$ as in \eqref{form-1-pl}.  Standard arguments show that (see for instance \cite{LST-1996})
\begin{equation}\label{E-1}
\|{ E}(t)v\|+\|{E}_h(t)v\| + t^{1-\alpha}\left(\|{\bar E}(t)v\|+\|{\bar E}_h(t)v\|\right)\leq c \|v\|\quad \forall v\in L^2(\Omega).
\end{equation}
 

Now let $e(t)=u_h(t)-u(t)$ denote the error at time $t$. Define the intermediate solution 
$v_h(t)\in V_h$, $t\geq 0$,  by
\begin{equation}\label{pr-3}
T_h\cop v_h +v_h = T_h f(u),\quad t>0 ,\quad  v_h(0)=u_{0h}.
\end{equation}
Then, by  splitting the error $e = (u_h-v_h)+(v_h-u)=:\eta+\xi$, and
 subtracting \eqref{pr-3} from  \eqref{pr-1}, we find that  $\xi$ satisfies 
\begin{eqnarray}\label{IVP-nn}
 T_h\cop \xi(t)+\xi(t)= (T_h-T)(f(u)-\cop u)(t),\quad t>0.
\end{eqnarray}
With 
\begin{equation} \label{rho}
\rho(t) := (T_h-T)(f(u)-\cop u)(t), 
\end{equation}
we thus obtain
\begin{equation} \label{eq:T_h}
T_h\cop \xi(t) +\xi(t) = \rho(t),\quad  t>0, \quad \xi(0)\in L^2(\Omega).
\end{equation}

Before proving the main result of this section, we  recall  the following lemma   which  generalizes  the classical Gronwall's inequality, see \cite{CTW-1992}.
\begin{lemma}\label{Gronwall}
Assume that $y$ is a nonnegative function in $L^1(0,T_0)$ which satisfies
$$
y(t) \leq g(t)+\beta\int_0^t(t-s)^{-\alpha}y(s)\,ds\quad \mbox{ for } t\in (0,T_0],
$$
where $g(t)\geq 0$, $\beta\geq 0$, and $0<\alpha<1$. Then there exists a constant $C_{T_0}$ such that 
$$
y(t) \leq g(t)+ C_{T_0}\int_0^t(t-s)^{-\alpha}g(s)\,ds\quad \mbox{ for }  
t\in (0,T_0].
$$
\end{lemma}
Note that, by using \eqref{form-2},  \eqref{E-1} and the inequality
\begin{equation}\label{stab}
\|f(u(t))\|\leq \|f(u(t))-f(0)\|+\|f(0)\|\leq L\|u\|+\|f(0)\|, \quad t\geq 0,
\end{equation}
Lemma \ref{Gronwall} implies that
$\|u(t)\|\leq c(\|u_0\|+\|f(0)\|)$ for $t\geq 0$ with $c=c(\alpha,L,T_0)$.
Now we are ready to prove an error estimate for problem \eqref{pr-1}. Here $\rho$ is given by \eqref{rho} and 
$\tilde \rho(t):=\int_0^t\rho(s)\,ds$.
\begin{lemma}\label{lem:main-1}
Let $u$  and $u_h$ be the solutions of \eqref{pr-1} and \eqref{pr-2}, respectively. Assume that 
$T_h(u_0-u_{0h})=0$. Then, for $t>0$,  
\begin{equation} \label{eq:error-1}
\|e(t)\|\leq c \left(G(t)+\int_0^t(t-s)^{\alpha-1}G(s)\,ds\right),
\end{equation}
where
\begin{equation} \label{xi-1}
G(t)= t^{-1}\sup_{s\leq t} ( \|\tilde\rho(s)\|+s\|\rho(s)\|+ s^2\|\rho_t(s)\|),
\end{equation}
and  $c$ is independent of $h$.
\end{lemma}
\begin{proof} First we derive a bound for the difference between $u$ and the intermediate solution $v_h$. Since $T_h\xi(0)=0$, an application of Lemma 3.5 in \cite{Karaa-2018} to  \eqref{pr-1} and \eqref{pr-3} yields $\|\xi(t)\|\leq c G(t)$, where $G$ is given by \eqref{xi-1}.
In view of the splitting $e =\eta+\xi$, it suffices to estimate $\|\eta\|$. Note that $\eta$ satisfies
$$
T_h\cop \eta(t) +\eta(t) = T_h(f(u_h)-f(u)),\quad  t>0, \quad \eta(0)=0.
$$
Hence, by  Duhamel's principle, 
$$
\eta(t)=\int_0^t\bar{E}_h(t-s)(f(u_h(s))-f(u(s)))\,ds, \quad t>0.
$$
Using the property of $\bar{E}_h$ in \eqref{E-1} and  condition \eqref{Lip}, we see that 
$$
\|\eta(t)\|\leq c L\int_0^t (t-s)^{\alpha-1}\|u_h(s)-u(s)\|\,ds, \quad t>0, 
$$
and thus
$$
\|e(t)\|\leq \|\xi(t)\|+c \int_0^t(t-s)^{\alpha-1}\|e(s)\|\,ds, \quad t>0. 
$$
An application of Lemma \ref{Gronwall} yields \eqref{eq:error-1}, which completes the proof.
\end{proof}

Clearly  the  error estimate  in Lemma \ref{lem:main-1} is meaningful provided that 
$G\in L^1(0,T)$. Recalling that $\rho=(T_h-T)\cL u$, we have by \eqref{ass-1}, $\|\partial_t \rho(t)\|\leq ch^\gamma \|\cL \partial_t u(t)\|$. Hence,  to achieve a $O(h^\gamma)$ order of convergence, we need to assume that $\partial_t u(t)\in \mathcal{D}(\cL)$ for $t\in (0,T_0]$. It turns out that, without additional conditions on initial data and  nonlinearity, this property, which holds in the linear case, does not  generalize to the semilinear problem. This remark equally applies to the semilinear parabalic problem, see the discussion in \cite[pp. 228]{thomee1997}.



\section{Error estimates without regularity assumptions}\label{sec:RA}

We shall present below an alternative derivation of the error bound  without a priori 
regularity assumptions on the exact solution. To do so, we first introduce the operator
$$S_h(z) =(z^{-\alpha}I+ T_h)^{-1}T_h-(z^{-\alpha}I+ T)^{-1}T.$$
Then $S_h$ satisfies the following property.
\begin{lemma}\label{lem-02} There holds
\begin{equation}\label{ass}
\|S_h(z)v\|\leq ch^{\gamma}|z|^{\alpha(1-\nu)}\|\cL^\nu v\|\quad \forall z\in\Sigma_\theta,\quad \nu\in [0,1].
\end{equation}
\end{lemma}
\begin{proof}
Using the identity  \eqref{expand}, we verify that
\begin{eqnarray*}
S_h(z) & = & z^{-\alpha}(z^{-\alpha} I+T)^{-1}\left[(z^{-\alpha} I+T_h)-(z^{-\alpha} I+T)\right] 
(z^{-\alpha} I+T_h)^{-1}\\
& = & z^{-\alpha}(z^{-\alpha} I+T)^{-1}(T_h-T)(z^{-\alpha} I+T_h)^{-1}.
\end{eqnarray*}
Then, by  \eqref{res1} and \eqref{ass-1}, 
$$ 
\|S_h(z)v\|\leq c \|(T_h-T)(z^{-\alpha} I+T_h)^{-1}v\|\leq ch^{\gamma} |z|^{\alpha}\|v\|.
$$
This shows \eqref{ass} for $\nu=0$. 
For   $\nu=1$, i.e., $v\in D(\cL)$, we have $T\cL v=v$. Then, by \eqref{ass-1} and \eqref{res2}, we get 
$$ \|S_h(z)v\|  \leq c  \|(T_h-T)(z^{-\alpha} I+T)^{-1}T\cL v\|\leq ch^{\gamma} \|\cL v\|. $$
The desired estimate \eqref{ass-1}  follows now by interpolation.
\end{proof}

We further introduce the following operators:  $F_h(t)=E_h(t)-E(t)$ and $\bar F_h(t)=\bar E_h(t)-\bar E(t)$. Then, by Lemma \ref{lem-02}, 
\begin{equation} \label{0-p0}
\|\bar F_h(t)v\|\leq ch^\gamma\int_{\Gamma_{\theta,1/t}}e^{Re(z)t}\,|dz|\;\|v\|\leq ct^{-1}h^\gamma\|v\|.
\end{equation}
Similarly, based on  \eqref{ass},  the following estimate 
\begin{equation} \label{0-p}
\|F_h(t)v\|\leq ct^{-\alpha(1-\nu)} h^\gamma \|\cL^\nu v\|, \quad \nu=0,1,
\end{equation}
holds for  $F_h(t)$. Now we are ready to prove a nonsmooth data error estimate. Here and the throughout the paper, 
$\eh =|\ln h|$ if $\nu=0$ and $1$ otherwise.
\begin{theorem}\label{thm:semi} 
Let $u_0\in \mathcal{D}(\cL^\nu)$, $\nu\in [0,1]$.
Let $u$ and $u_h$ be the solutions defined by \eqref{form-2} and  \eqref{form-1}, respectively, with $u_{0h}=P_hu_0$.
Then there is a constant $c=c(r,L,T_0)$, where $r\geq \|\cL^\nu u_0\|+\|f(0)\|$, such that 
\begin{equation} \label{01-bb-s}
\|u_h(t)-u(t)\|\leq ch^\gamma \eh t^{-\alpha(1-\nu)},\quad t\in (0,T_0].
 \end{equation}
\end{theorem}
\begin{proof} Recall that $e=u_h-u$. From \eqref{form-1} and \eqref{form-2}, we get after rearrangements 
\begin{equation}\label{intt}
e(t)= F_h(t)u_0+\int_0^t \bar{E}_h(t-s)[f(u_h(s))-f(u(s))]\,ds+\int_0^t \bar{F}_h(t-s) f(u(s))\,ds.
\end{equation}
The last term in \eqref{intt} can be written as $I+II$ where 
$$I = \int_0^t \bar{F}_h(t-s) (f(u(s))-f(u(t)))\,ds \quad \mbox{and} \quad II=\left(\int_0^t \bar{F}_h(t-s) \,ds\right)f(u(t)).$$
For $\nu \in (0,1]$, we use   \eqref{0-p0}, \eqref{Lip} and the property $u\in C^{\alpha \nu}([0,T_0],L^2(\Omega))$ to get
$$ 
\|I\|\leq ch^\gamma\int_0^t(t-s)^{-1}(t-s)^{\alpha \nu}\,ds \leq ch^\gamma t^{\alpha \nu}.$$
To estimate $II$, we introduce the operator  
$\widetilde E(t)=\frac{1}{2\pi i}\int_{\Gamma_{\theta,\delta}}e^{zt} z^{-1-\alpha} S_h(z)\,dz.$
Then  $\widetilde E'(t)=\bar F(t)$  and  $\|\widetilde E(t)\| \leq ch^\gamma$ for all $t\geq 0$ since $\|S_h(z)\|\leq ch^\gamma|z|^{\alpha}$. Hence,  $\|II\|\leq \|f(u(t))\| \,\|\widetilde E(t)-\widetilde E(0)\|\leq c h^\gamma$ as $\|f(u)\|$ is bounded, see \eqref{stab}.
 Now, using the properties of  
$\bar E_h$ and $F_h$ in \eqref{E-1} and \eqref{0-p}, respectively, we obtain 
\begin{equation}
\|e(t)\|\leq ct^{-\alpha(1-\nu)} h^\gamma \|\cL^\nu u_0\|+c\int_0^t(t-s)^{\alpha-1}\|e(s)\|\,ds+ch^\gamma + ch^\gamma t^{\alpha \nu}.
\end{equation}
The desired estimate follows now by applying Lemma \ref{Gronwall}. 
To establish the estimate  for $\nu=0$, we 
follow the same arguments presented in the proof of Theorem 4.4 in \cite{Karaa-2018b}.
\end{proof}

As an immediate application of Theorem \ref{thm:semi}, consider the standard conforming Galerkin FEM with $V_h\subset H_0^1(\Omega)$ consists of piecewise linear functions on a shape-regular triangulation with a mesh parameter $h$. 
Let $T_h:L^2(\Omega)\rightarrow V_h$ be the solution operator of the discrete  problem:
\begin{equation}\label{T_h-ngg-cc}
a(T_hf, v)=(f,v)\quad \forall v \in V_h,
\end{equation}
where $a(\cdot,\cdot)$ is the bilinear form associated with the elliptic operator $\cL$.
Then, $T_h$ is selfadjoint, positive semidefinite  on $L^2(\Omega)$ 
and positive definite  on $V_h$, see \cite{BSTW-1977}, and satisfies \eqref{ass-1} with $\gamma=2$.
Thus, by Theorem \ref{thm:semi},
\begin{equation} \label{01-new-cc}
 \|u_h(t)-u(t)\| \leq c h^2 \eh t^{-\alpha(1-\nu)}, \quad t >0,
\end{equation} 
for $\nu\in [0,1]$. This  improves  the following estimate 
\begin{equation} \label{01-old-cc}
\|u_h(t)-u(t)\|\leq ch^2\left(t^{-\alpha(1-\nu)}+\max(0,\ln (t^{\alpha(1-\nu)}/h^2))\right),\quad t>0,
 \end{equation}
established in \cite{Karaa-2018b}. We notice that the logarithmic factor is also present in the parabolic case when $\nu=0$, see \cite[Theorem 1.1]{JLTW-1987}.



As a second example, we show that the present semidiscrete error analysis extends 
 to the following  multi-term time-fractional diffusion problem:
\begin{equation}\label{a-m}
P(\partial_t) u  + \cL u=f(u)\; \mbox{ in } \Omega\times (0,T_0],\quad u(0)=u_0\; \mbox{ in } \Omega, \quad u=0 \; \mbox{ on } \partial \Omega\times (0,T_0],
\end{equation}
where  the multi-term differential operator $P(\partial_t)$ is defined by
$
P(\partial_t) = \partial_t^{\alpha}+\sum_{i=1}^m b_i\partial_t^{\alpha_i}
$
with $0<\alpha_m\leq\cdots\leq \alpha_1\leq \alpha<1$ being the orders of the fractional Caputo derivatives, and  $b_i>0$, $i=1,\ldots,m$.
This model  was derived to improve the modeling accuracy of the single-term model 
\eqref{a}  for describing anomalous diffusion. An inspection of the proof of the
Theorem \ref{thm:semi} reveals that its main arguments are based on the bounds derived for the operators $F_h$, $\bar{E}_h$ and  $\bar{F}_h$. Following \cite{JLZ2015}, one can verify that these operators satisfy the same bounds as in the single-term case. This readily implies that the estimate  \eqref{01-bb-s} remains valid for the multi-term diffusion problem  \eqref{a-m}.


\begin{remark}\label{parabolic} In the parabolic case, a singularity in time appears which has the same form as in \eqref{0-p0}. Hence, the estimate \eqref{01-bb-s} remains valid when $\alpha=1$. 
 
\end{remark}  

\begin{remark}\label{Ru}
For smooth initial data $u_0\in \mathcal{D}(\cL)$, the estimate \eqref{01-bb-s}  still holds for the choice $u_{0h}=R_hu_0$. Indeed, we have
$$
E_h(t)R_h u_0 - E(t)u_0 = E_h(t)(R_h u_0 - u_0) + (E_h(t)u_0- E(t)u_0).
$$
By the stability of the operator  $E_h(t)$,
$$
 \|E_h(t)(R_h u_0 - u_0)\|\leq   c\|R_h u_0 - u_0\|\leq ch^\gamma\|\cL u_0\|.
$$
Then we reach our conclusion by following the arguments in the proof of Theorem \ref{thm:semi}. 
\end{remark}


\section{Fully discrete schemes}\label{sec:FD}

This section is devoted to the analysis of a fully discrete scheme for  problem \eqref{pr-2} based on a convolution quadrature (CQ) generated by the backward Euler method, using the framework developed in 
\cite{LST-1996, Lubich-2006}.
Divide the time interval $[0,T_0]$ into $N$ equal subintervals with a time step size $\tau=T_0/N$, and let $t_j=j\tau$. 
The convolution quadrature \cite{Lubich-1986} refers to an approximation of any function of 
the form $k\ast\varphi$ as
$$
(k\ast\varphi) (t_n):=\int_0^{t_n}k(t_n-s)\varphi(s)\,ds\approx \sum_{j=0}^n \beta_{n-j}(\tau) \varphi(t_j),
$$
where the weights $\beta_j=\beta_j(\tau)$ are computed from the Laplace transform $K(z)$  of 
$k$ rather than the kernel $k(t)$. 
With $\partial_t$ being time differentiation,  define $K(\partial_t)$ as the operator 
of (distributional) convolution with the kernel $k$: $K(\partial_t)\varphi=k\ast \varphi$ for a 
function $\varphi(t)$ with suitable smoothness.
Then a convolution quadrature will approximate $K(\partial_t)\varphi$ by a discrete convolution 
$K( \partial_\tau)\varphi$ at $t=t_n$  as
$
K( \partial_\tau)\varphi(t_n) = \sum_{j=0}^n \beta_{n-j}(\tau) \varphi(t_j),
$
where the quadrature weights $\{\beta_j(\tau)\}_{j=0}^{\infty}$ are determined by the generating 
power series
$
\sum_{j=0}^\infty  \beta_j(\tau) \xi^j=K(\delta(\xi)/\tau)
$
with $\delta(\xi)$ being a rational function, chosen as the quotient of the generating polynomials of 
a stable and consistent linear multistep method. For the backward Euler  method,  $\delta(\xi)=1-\xi$.

An important property of the convolution quadrature is that it maintains some relations of the continuous convolution. 
For instance, the associativity of convolution is valid for the convolution 
quadrature  \cite{Lubich-2006} such as
\begin{equation}\label{s1}
K_2( \partial_\tau) K_1( \partial_\tau)=K_2 K_1( \partial_\tau)\quad 
\text{ and }\quad 
K_2( \partial_\tau)(k_1\ast\varphi)=(K_2( \partial_\tau)k_1)\ast\varphi.
\end{equation}

In the following lemma, we state an interesting result on the error of the convolution quadrature \cite[Theorem 5.2]{Lubich-1988}.
\begin{lemma}\label{lem:Lubich}
Let $G(z)$ be  analytic in the sector $\Sigma_\theta$ and such that
\begin{equation*}\label{s3}
\|G(z)\|\leq c|z|^{-\mu}\quad \forall z\in \Sigma_\theta,
\end{equation*}
for some real $\mu$ and $c$. Then, for $\varphi(t)=ct^{\sigma-1}$, the convolution quadrature based on the backward Euler method satisfies
\begin{equation}\label{s4}
\|G(\partial_t)\varphi(t) - G( \partial_\tau)\varphi(t)\|  \leq \left\{
\begin{array}{ll}
C t^{\mu+\sigma-2} \tau, & \sigma\geq 1\\
C t^{\mu-1} \tau^\sigma, & 0< \sigma\leq 1.
\end{array} \right.
\end{equation}
\end{lemma}


Upon using the relation between the Riemann-Liouville derivative denoted by $\op$ and the Caputo derivative $\cop$, the semidiscrete scheme \eqref{pr-2} can be rewritten  as
\begin{equation}\label{Ph-1RL}
T_h\op (u_h -u_{0h}) +u_h = T_hf(u_h),\quad t>0,\quad u_h(0)=u_{0h}.
\end{equation}
Thus, the proposed backward Euler CQ scheme is to seek $U_h^n\in V_h$, $n\geq 0$, such that
\begin{equation}\label{BE-1a}
T_h\dop (U_h^n -U_h^0) +U_h^n = T_hf(U_h^n), \quad  n\geq 1,\quad U_h^0=u_{0h}.
\end{equation}


\subsection{The linear case}
We begin by investigating the time discretization of the  inhomogeneous linear problem
\begin{equation}\label{C1}
T \cop u +u = Tf(t),\quad t>0,\quad u(0)=u_0,
\end{equation}
with a semidiscrete solution $u_h(t)\in V_h$  satisfying
\begin{equation}\label{C2}
T_h \cop u_h +u_h = T_hf(t),\quad t>0,\quad u_h(0)=P_h u_{0}.
\end{equation}
The fully discrete solution $U_h^n\in V_h$  is defined by 
\begin{equation}\label{C3}
T_h\dop (U_h^n -U_h^0) +U_h^n = T_hf(t_n), \quad n\geq 1,\quad U_h^0=P_h u_{0}.
\end{equation}
Then we establish the following result.
\begin{theorem}\label{thm:BE}
Let $u_0\in \mathcal{D}(\cL^\nu)$, $\nu\in [0,1]$.
Let $u$ and $U_h^n$ be the solutions of  $\eqref{C1}$ and $\eqref{C3}$, respectively, with $f=0$. 
 Then, for $t_n>0$, 
\begin{equation}\label{s9}
\|U^n_h-{u}(t_n)\|\leq c (\tau t_n^{\alpha \nu-1}+ h^\gamma t_n^{-\alpha(1-\nu)})\|\cL^\nu u_0\|.
\end{equation}
\end{theorem}
\begin{proof} We first notice that  $\|u_h(t_n)-{u}(t_n)\|\leq c t_n^{-\alpha(1-\nu)}h^\gamma\|\cL^\nu u_0\|$, which follows from  Theorem \ref{thm:semi} in the case $f=0$.  In order  to estimate $\|U^n_h-{u}_h(t_n)\|$, apply $\partial_t^{-\alpha}$ and $\partial_\tau^{-\alpha}$ to \eqref{C2} and \eqref{C3},
respectively, use the associativity of convolution and the property $T_hP_h=T_h$,
to deduce that 
\begin{equation}\label{s7-pt}
U^n_h-{u}_h(t_n)= - \left( G(\partial_\tau)- G(\partial_t)\right)u_{0},
\end{equation}
where $G(z)=(z^{-\alpha} I+T_h)^{-1}T_h$.
We recall that, by \eqref{res1},  $\| G(z)\|\leq c\;  \forall z\in \Sigma_\theta.$ Then,  Lemma \ref{lem:Lubich}  (with $\mu=0$,  $\sigma=1$)   and the $L^2$-stability of $P_h$ yield
\begin{equation}\label{s8}
\|U^n_h-{u}_h(t_n)\|\leq c \tau t_n^{-1}\|u_0\|.
\end{equation}
For $u_0\in \mathcal{D}(\cL)$,  consider first the choice $u_h(0)=R_h u_0$. Recalling that $R_h=T_h \cL$, we  use  the identity $G(z)=I-z^{-\alpha}(z^{-\alpha} I+T_h)^{-1}$ to get
$$U^n_h-{u}_h(t_n)= -\left( \bar{G}(\partial_\tau)- \bar{G}(\partial_t)\right)\cL v,$$ 
where $\bar{G}(z)=z^{-\alpha}(z^{-\alpha} I+T_h)^{-1}T_h$.
Since $\|\bar{G}(z)\|\leq c |z|^{-\alpha}$ $\forall z\in \Sigma_\theta$, an application of Lemma \ref{lem:Lubich}  (with $\mu=\alpha$,  $\sigma=1$) yields 
\begin{equation}\label{s8-b}
\|U^n_h-{u}_h(t_n)\|\leq c \tau t_n^{\alpha-1}\|\cL v\|.
\end{equation}

For the choice $u_0=P_hv$, we split the new error  
$$ G(\partial_\tau)u_0- G(\partial_t)u_0= G(\partial_\tau)(u_0-R_h u_0) + G(\partial_\tau)R_hu_0-G(\partial_t)u_0.$$
By the stability of the  discrete scheme, the estimate \eqref{s8-b} and remark \ref{Ru}, we deduce that
$
\|G(\partial_\tau)u_0- G(\partial_t)u_0\|\leq c(h^\gamma+\tau t_n^{\alpha-1})\|\cL u_0\|.
$
This shows  \eqref{s9}  when $\nu=1$. Finally, for $\nu \in (0,1)$, the estimate \eqref{s9} follows by interpolation. 
\end{proof}

Now we  prove an error estimate when $f\neq 0$ but with $u_0=0$. 
\begin{theorem}\label{thm:BE-n}
Let $u$ and $U_h^n$ be the solutions of  $\eqref{C1}$ and $\eqref{C3}$, respectively, with $u_0=0$.
If $f$ satisfies $|f(t)-f(s)|\leq c|t-s|^\theta$ $\forall t,s\in \mathbb{R}$ with some $\theta \in (0,1)$ and  $\int_0^t (t-s)^{\alpha-1}\|f'(s)\|ds<\infty $ $\forall t\in(0,T_0]$, 
then, for $t_n>0$, 
\begin{equation}\label{s12-n}
\|U^n_h-u(t_n)\|\leq  c h^\gamma+c 
\left(\tau t_n^{\alpha-1}\|f(0)\|+\tau\int_0^{t_n} (t_n-s)^{\alpha-1}\|f'(s)\|ds\right).
\end{equation}
\end{theorem}
\begin{proof} Using the assumptions on $f$ in the proof of Theorem \ref{thm:semi} we conclude that $\|u_h(t)-u(t)\|\leq  c h^\gamma$. From \eqref{C2} and \eqref{C3}, the semidiscrete and fully discrete solutions are now given by $u_h=\bar G(\partial_t)f$ and $U_h^n=\bar G(\partial_\tau)f$, respectively,  where $\bar G(z)$ is defined in the previous proof.  Using the expansion $f(t)=f(0)+(1\ast f')(t)$ and the 
second relation in \eqref{s1}, we find that
\begin{equation*}\label{s7-knn}
U^n_h-{u}_h(t_n)= ( \bar G(\partial_\tau)- \bar G(\partial_t))f(0)+(( \bar G(\partial_\tau)-  \bar G(\partial_t))1)\ast  f'(t_n)=:I+II.
\end{equation*}
By Lemma \ref{lem:Lubich} (with $\mu=\alpha$ and $\sigma=1$), we have
$$
\|I\|\leq  c\tau  t_n^{\alpha-1}\|f(0)\|.
$$
For the second term, Lemma \ref{lem:Lubich} yields
$$
\|II\|\leq  \int_0^{t_n} \|(( \bar G(\partial_\tau)-  \bar G(\partial_t))1)(t_n-s) f'(s)\| \leq 
c\tau \int_0^{t_n} (t_n-s)^{\alpha-1} \|f'(s)\|ds,
$$
which completes the proof of \eqref{s12-n}.
\end{proof}



\subsection{The semilinear case}
We consider now  the time approximation of the nonlinear problem \eqref{pr-1} and prove related error estimates. The time-stepping scheme is now defined as follows: find  $U_h^n\in V_h$, $n\geq 0$,   such that
\begin{equation}\label{C4}
T_h\dop (U_h^n -U_h^0) +U_h^n = T_hf(U_h^n), \quad n\geq 1, \quad U_h^0=P_h u_{0}.
\end{equation}
This scheme is written in an expanded form as 
$$
T_h(U_h^n-U_h^0)+\tau^{\alpha}\sum_{j=0}^n q_{n-j}^{(\alpha)} U_h^j= \tau^{\alpha}T_h \sum_{j=0}^n q_{n-j}^{(\alpha)} f(U_h^j),
$$
where 
$ 
q_{j}^{(\alpha)} = (-1)^{j}
\left(\begin{array}{c}
-\alpha\\
j
\end{array}\right),
$
see \cite{Lubich-1988}.
%
Since  $q_0^{(\alpha)}=1$,  the  implicit equation for $U_h^n$ is of the form
\begin{equation}\label{implicit-1}
(\tau^\alpha I+ T_h)U_h^n=T_h\zeta +\eta +\tau^\alpha T_h f(U_h^n),\quad n\geq 1,
\end{equation}
where $\zeta,\eta\in V_h$. The solvability of  \eqref{implicit-1} is discussed below.
\begin{proposition}
Under the restriction $\tau^\alpha M L<1$, the nonlinear system \eqref{implicit-1} has a unique solution $U_h^n\in V_h$ for every $n\geq 1$.
\end{proposition}
\begin{proof}
The solvability of \eqref{implicit-1} is equivalent to the existence of a fixed point for the mapping 
$S_h:V_h \to V_h$ defined by
\begin{equation}\label{cont-1}
S_h(v)=(\tau^\alpha I+ T_h)^{-1}(T_h\zeta +\eta +\tau^\alpha T_h f(v)).
\end{equation}
Recalling that $\|(\tau^\alpha I+ T_h)^{-1}T_h\|\leq M$, we have $\forall v,w\in V_h$, 
\begin{eqnarray*}
\|S_h(w)-S_h(v)\| & \leq & \tau ^\alpha \|(\tau^\alpha I+ T_h)^{-1}T_h(f(v)-f(w))\|\\
& \leq & \tau ^\alpha M \|f_h(v)-f_h(w)\|\\
& \leq & \tau ^\alpha M L \|v-w\|.
\end{eqnarray*}
Hence, $S_h$ is a contraction if $\tau^\alpha M L<1$, and therefore, \eqref{cont-1} has a unique fixed point in $V_h$ which is also the unique  solution of \eqref{implicit-1}.
\end{proof}

To investigate the stability of the scheme \eqref{C4},  we rewrite it in the form
\begin{equation}\label{semi-1b}
(\dopm +T_h)  U_h^n = T_h U_h^0 +  \dopm T_h  f(U_h^n).
\end{equation}
Noting that $U_h^n$ depends linearly and boundedly on $U_h^0$, $ f(U_h^j)$, $0\leq j\leq n$, 
we deduce the existence of linear and bounded operators $P_n$ and $R_n:V_h\to V_h$, $n\geq 0$, such that  $U_h^n$ is represented  by
\begin{equation}\label{semi-1c}
U_h^n = P_n U_h^0  + \tau \sum_{j=0}^n R_{n-j}  f(U_h^j),
\end{equation}
see \cite[Section 4]{Lubich-2006}.
In view of  \eqref{semi-1b}, the operators $\tau R_n$, $n\geq 0$, are the convolution quadrature weights corresponding to the Laplace transform
$K(z)=z^{-\alpha}(z^{-\alpha}I+ T_h)^{-1}T_h$, i.e., they are the coefficients in the series expansion
$$
\tau\sum_{j=0}^\infty R_j\xi^j=K((1-\xi)/\tau).
$$
Since $\|K(z)\|\leq c|z|^{-\alpha}$, an application of Lemma 3.1 in \cite{Lubich-2006} with $\mu=\alpha$, shows that there is a constant $B>0$, independent of $\tau$, such that 
\begin{equation}\label{R_n}
\|R_n\|\leq B t_{n+1}^{\alpha-1},\quad n=0,1,2,\ldots.
\end{equation}
Based on this bound, we show that the scheme  \eqref{C4} is stable in $L^2(\Omega)$.
\begin{proposition}
Under the restriction $\tau^\alpha B L<1$, there exists a constant $C$ independent of $h$ and $\tau$ such that
\begin{equation}\label{stab-discrete}
\|U_h^n\|\leq C(\|U_h^0\|+\|f(u_0)\|), \quad n\geq 1.
\end{equation}
\end{proposition}
\begin{proof} Taking $L^2(\Omega)$-norms in \eqref{semi-1c} and using \eqref{R_n}, we deduce that
$$
\|U_h^n\| \leq c\|U_h^0\|+ \tau B\sum_{j=0}^n t_{n-j+1}^{\alpha-1}\|f(U_h^j)\|.
$$
Since, $\|f(U_h^j)\|\leq L \|U_h^j\|+\|f(0)\|$, 
$$
\|U_h^n\| \leq c\|U_h^0\|+ \frac{B}{\alpha}T^\alpha \|f(0)\|+ \tau B L\sum_{j=0}^n t_{n-j+1}^{\alpha-1}\|U_h^j\|.
$$
With the assumption that $\tau^{\alpha} B L<1$, a generalized discrete Gronwall's lemma (see, e.g., Theorem 6.1 in  \cite{DM-1986}) readily implies \eqref{stab-discrete}.
\end{proof}

Now we are ready to prove the main result of this section.
%
%
%
\begin{theorem}\label{thm:fully-1} Let $u_0\in \mathcal{D}(\cL^\nu)$, $\nu\in(0,1]$. 
Let $u$ be the solution of \eqref{pr-1}. 
Then there exists $\tau_0>0$ such that, for $0<\tau<\tau_0$, the numerical solution $U_h^n$  given by \eqref{C4} is uniquely defined and satisfies 
\begin{equation} \label{estimate-1a}
\|U_h^n-u(t_n)\|\leq c(\tau t_n^{\alpha\nu-1}+h^\gamma t_n^{\alpha(\nu-1)}),\quad t_n>0.
\end{equation} 
\end{theorem}
\begin{proof} Select $\tau_0$ such that $\tau_0^\alpha ML<1$. 
Then, for $0<\tau<\tau_0$, the discrete solution $U_h^n\in V_h$  is well defined. 
Let $v_h^n$, $n\geq 0$,  be the intermediate discrete solution defined by 
\begin{equation} \label{vv}
T_h \dop(v_h^n-v_h^0)+v_h^n=T_hf(u(t_n)),\quad n\geq 1, \quad v_h^0=U_h^0.
\end{equation} 
Then \eqref{vv} can be viewed  as a fully discretization of \eqref{pr-1} with a given right-hand side function 
$f(u)$. Hence, by applying Theorems \ref{thm:BE} and \ref{thm:BE-n}, and using the bound  $\|\partial_t u(s)\|\leq cs^{\alpha\nu-1}$, we deduce that
\begin{equation} \label{vv-1}
\begin{split} 
\|u(t_n)-v_h^n\|\leq & c\tau t_n^{\alpha\nu-1}+c h^\gamma t_n^{\alpha(\nu-1)}+ch^\gamma +c\tau t_n^{\alpha-1}\|f(u(0))\|\\
&+c\tau\int_0^{t_n}(t_n-s)^{\alpha-1}\|\partial_t f(s,x,u(s))+\partial_u f(s,x,u(s))\partial_t u(s)\|\,ds\\
\leq & c(\tau t_n^{\alpha\nu-1}+h^\gamma  t_n^{\alpha(\nu-1)}+ch^\gamma +\tau t_n^{\alpha-1}+\tau t_n^{\alpha+\alpha\nu-1})\\
\leq & c(\tau t_n^{\alpha\nu-1}+h^\gamma t_n^{\alpha(\nu-1)}).
\end{split} 
\end{equation} 
On the other hand, expressing $v_h^n$ in terms of the data, through the discrete operators $P_n$ and $R_n$, as in \eqref{semi-1c}, we obtain
\begin{equation}\label{semi-1d}
v_h^n = P_n U_h^0  + \tau \sum_{j=0}^n R_{n-j} f(u(t_j)).
\end{equation}
Thus, in view of \eqref{semi-1c} and \eqref{semi-1d}, we have for $0<t_n\leq T_0$,
\begin{eqnarray*}
U_h^n-u(t_n) & = & U_h^n-v_h^n+v_h^n-u(t_n) \\
& = & v_h^n-u(t_n) +\tau \sum_{j=0}^n R_{n-j} (f(U_h^j)-f(u(t_j))).
\end{eqnarray*}
By \eqref{Lip} and the estimate in \eqref{R_n} for $R_n$, we obtain
$$
\|U_h^n-u(t_n)\|\leq  \|v_h^n-u(t_n)\|+ \tau  LB \sum_{j=1}^n t_{n-j+1}^{\alpha-1} \|U_h^j-u(t_j)\|+
\tau LB t_{n+1}^{\alpha-1} \|P_hu_0-u_0\|.
$$
Using the estimate \eqref{vv-1} and making the additional assumption $\tau_0^\alpha LB<1$, we now apply a generalized Gronwall's inequality, see Lemma 5.1 in \cite{Karaa-2018b}, to finally derive \eqref{estimate-1a}.
\end{proof}

\section{Galerkin type approximations}\label{sec:Appli}
In this section,  we apply our analysis to approximate the solution of \eqref{a} by 
general Galerkin type methods and derive optimal $L^2(\Omega)$-error estimates  for  cases with smooth and nonsmooth initial data. For a general setting, we assume that we are given a bilinear form $a_h:V_h\times V_h\to \mathbb{R}$ which has the following property:

{\bf Property A.} $a_h(\cdot,\cdot)$ is symmetric positive definite, and the discrete problem
\begin{equation}\label{T_h}
a_h(T_hf, \chi)=(f,\chi)\quad \forall \chi \in V_h
\end{equation}
defines a linear operator $T_h:L^2(\Omega)\rightarrow V_h$  satisfying the estimate \eqref{ass-1}. 

The solution of the continuous problem \eqref{a} will be approximated through the semidiscrete problem:  find  $u_h(t)\in V_h$ such that
\begin{equation} \label{semi}
(\cop u_h,\chi)+ a_h(u_h,\chi)= (f(u_h),\chi) \quad \forall \chi\in V_h, \quad t>0,
\end{equation}
with $u_h(0)=P_hu_0$. We recall that $P_h:L^2(\Omega)\to V_h$ is the $L^2$-projection onto $V_h$. Next  we  define the fully discrete scheme based on the backward Euler CQ method as follows: with $U_h^0=P_hu_0$, find $U_h^n\in V_h$, $n\geq 1$, such that
\begin{equation} \label{H1}
(\dop U_h^n,\chi)+ a_h(U_h^n,\chi)=  (\dop U_h^0,\chi)+ (f(U_h^n),\chi)\quad \forall \chi\in V_h.
\end{equation}
Then we have the following result. 
\begin{theorem}\label{thm:BE-G}
Let ${u}$ be the solution of problem $(\ref{a})$ with $u_0\in \mathcal{D}(\cL^\nu)$, $\nu\in (0,1]$.  Let  $U^n_h$ be the solution of problem  $\eqref{H1}$ with $U_h^0=P_hu_0$. 
Assume that $a_h(\cdot,\cdot)$ satisfies Property A.
Then there holds:
\begin{equation}\label{s9-G}
\|U^n_h-{u}(t_n)\|\leq c (\tau t_n^{\alpha\nu-1} + h^\gamma t_n^{\alpha(\nu-1)}),\quad t_n>0,
\end{equation}
where $c$ is independent of $h$ and $\tau$.
\end{theorem}
\begin{proof}
Since $a_h(\cdot,\cdot)$ is symmetric, the operator $T_h$ is selfadjoint and  positive semidefinite  on $L^2(\Omega)$: for all $f,g \in L^2(\Omega)$
$$(f,T_hg)=a_h(T_hf,T_hg)=(T_hf,g) \quad \mbox{ and }\quad (f,T_hf)=a_h(T_hf,T_hf)\geq 0.$$
If $f\in V_h$ and $T_hf=0$, then \eqref{T_h} implies that $f=0$, that is $T_h$ is positive definite on  $V_h$. Further, as from  \eqref{T_h}, we have $T_h=T_hP_h$. Hence, $T_h$ satisfies
all the conditions stated  in section \ref{sec:AB}.
In view of \eqref{T_h}, the fully discrete scheme \eqref{H1} is equivalently written as 
\begin{equation} \label{H2}
T_h \dop (U_h^n-U_h^0)+U_h^n= T_hf(U_h^n), \quad n\geq 1,\quad U_h^0=P_hu_0.
\end{equation}
Thus, the desired estimate is a direct consequence of Theorem\,\ref{thm:fully-1}.
\end{proof}

We notice that Property A is quite standard and holds for a large class of Galerkin type approximation methods, including  FE and spectral methods. For the spectral methods, one may choose the notation $V_N$ instead of $V_h$ and 
replace the estimate \eqref{ass-1} by $\|T_hf-Tf\|\leq cN^{-\gamma}\|f\|\; \forall f\in L^2(\Omega).$
Further, since we are not imposing restrictions  on the space $V_h$, our analysis applies to  nonconforming space approximations, such as the early method by Nitsche
\cite{Nitsche1971} and the method by Crouzeix and Raviart \cite{CR-1973}.
Recent examples of nonconforming  methods include  discontinuous Galerkin (DG) FE methods.
An interesting case is the symmetric  DG interior penalty method, 
 see \cite{ABCM-2}.
Here the discrete bilinear form $a_h$ for the Laplacian operator is given by
\begin{eqnarray} \label{SDG}
a_h(u,v)&:=&\sum_{K\in \mathcal T_h}\int_K c^2\nabla u\cdot\nabla v\,dx-
\sum_{F\in {\mathcal F}_h}\int_F [[u]]\cdot\{\{c^2\nabla v\}\}\,ds\nonumber\\
&&-\sum_{F\in {\mathcal F}_h}\int_F [[v]]\cdot\{\{c^2\nabla u\}\}\,ds
+\sum_{F\in {\mathcal F}_h}\rho h_F^{-1}\int_F c^2[[v]]\cdot[[u]]\,ds,
\end{eqnarray}
where $\mathcal{T}_h=\{K\}$ is a regular partition of $\bar\Omega$ and  
$\mathcal{F}_h=\{F\}$ is the set of all interior and boundary edges or faces of $\mathcal{T}_h=\{K\}$.
The last three terms in \eqref{SDG} correspond to jump and flux terms at
element boundaries, with $h_F$ denoting the diameter of the edge or the face $F$,
see \cite{ABCM-2} for more details.  The parameter $\rho>0$ is the interior
 penalty stabilization parameter that has to be chosen sufficiently large, independent of the mesh size. The bilinear form $a_h$ is clearly symmetric and satisfies property A with $\gamma=2$. Our analysis  extends to other symmetric spatial DG methods as long as property A is satisfied.

\section {Mixed FE methods} \label{sec:Mixed}

We consider now the mixed form of the problem (\ref{a}) and establish {\it a priori} error estimates for smooth and nonsmooth initial data.  For the sake of simplicity, we choose $\cL=-\Delta$ and $\Omega\subset \mathbb{R}^2$. One notable advantage of the mixed FEM is 
that it approximates the solution $u$ and its gradient simultaneously, resulting in a high convergence rate for the gradient.


By introducing the new variable $\bs=\nabla u$, the problem can be formulated as
$$
\cop u- \nabla \cdot \bs =f(u), \qquad \bs=\nabla u, \qquad u=0 \; \mbox{ on } \partial\Omega,
$$
with $u(0)=u_0$. Let $H(div;\Omega)= \{\bs\in (L^2(\Omega))^2:\nabla\cdot\bs\in L^2(\Omega) \}$
be a Hilbert space equipped with norm $\|\bs\|_{\bV} =(\|\bs\|^2+\|\nabla\cdot\bs\|
^2)^{\frac{1}{2}}$.
Then, with  $V=L^2(\Omega)$ and $\bV= H(div;\Omega)$,  the weak mixed formulation of (\ref{a}) 
is defined as follows:  find $(u,\bs):(0,T_0]\to  V\times \bV$ such that
\begin{eqnarray}\label{w1-m}
(\cop u, \chi)- (\nabla\cdot \bs, \chi) &=& (f(u),\chi) \;\;\;\forall \chi \in V,\\
\label{w2-m}
(\bs, \bv) +  (u,\nabla\cdot \bv) &=& 0\ \;\;\;\forall \bv \in \bV,
\end{eqnarray}
with $u(0)=u_0$.  Note that the boundary condition $u=0$ on $\partial\Omega$ is implicitly contained in 
\eqref{w2-m}. By Green's formula, we formally obtain $\bs = \nabla u$ in $\Omega$ and $u=0$ on $\partial \Omega$.

For the mixed form of  problem \eqref{a}, a  few numerical studies are available, dealing only with the linear case. In \cite{CockburnMustapha2015}, the authors investigated a  hybridizable DG method for the space discretization. In \cite{2017}, a non-standard mixed FE method was proposed and analysed. Another related analysis for mixed method applied to the time-fractional Navier-Stokes equations was presented in \cite{XYZ-2017}. The convergence analyses in all these studies require  high regularity assumptions on the exact solution, which is not in general reasonable.  In the recent work \cite{Karaa-2018}, we investigated a mixed FE method for  \eqref{a} with $f=0$  and derived optimal error estimates for semidiscrete schemes with smooth and nonsmooth initial data.  
The estimates extend the results obtained for the standard linear parabolic problem \cite{JT-1981}.
In the present analysis, we shall avoid energy arguments as employed in \cite{Karaa-2018} due to the weak regularity of the solution.

For the   FE approximation of problem \eqref{w1-m}-\eqref{w2-m}, let ${\mathcal T}_h$ be a shape regular and quasi-uniform partition of the polygonal convex domain $\bar \Omega$ into triangles $K$ of diameter $h_K$. 
Further, let $V_h$ and $\bV_h$  
 be the Raviart-Thomas FE spaces \cite{RT-1977} of index $\ell\geq 0$ given respectively  by
 $$
 V_h=\{ w\in L^2(\Omega):\;w|_{K}\in P_{\ell}(K) \;\forall K\in {\mathcal T}_h\}
 $$
 and 
 $$
 \bV_h=\{ {\bf v} \in H(div,\Omega):\;{\bf v}|_{K}\in RT_{\ell}(K) \;\forall K\in {\mathcal T}_h\},
 $$
where $RT_{\ell}(K)=(P_{\ell}(K))^2+{\mathcal \mathbb x}P_{\ell}(K),$ ${\ell}\geq 0$. 
Other examples of mixed FE spaces may also be considered.
We shall restrict our analysis to the low order cases $\ell=0,1,$ as
high order Thomas-Raviart elements are not attractive due to the limited smoothing property of the time-fractional model, see \cite{Karaa-2018}.

For $(u,\bs)\in V\times \bV$, we define the intermediate mixed projection as the pair $(\tilde{u}_h,\tilde{\bs}_h)\in V_h\times \bV_h$ satisfying
\begin{eqnarray}\label{ppp-1}
(\nabla\cdot (\bs-\tilde{\bs}_h), \chi_h) &=& 0 \;\;\;\forall \chi_h \in V_h,\\
\label{ppp-2}
( (\bs-\tilde{\bs}_h), \bv_h) +  (u-\tilde{u}_h,\nabla\cdot \bv_h) &=& 0 \;\;\;\forall \bv_h \in \bV_h.
\end{eqnarray}
Then,  with $(u,\bs)=(u,\nabla u)$, the following estimates hold  \cite[Theorem 1.1]{JT-1981}: 
\begin{equation}\label{est-1}
\|u-\tilde{u}_h\|\leq C h^{1+\ell}\|\cL u\|,\quad \|\bs-\tilde{\bs}_h\|\leq Ch\|\cL u\|,\quad \ell=0,1.
\end{equation}

For a given function $f\in L^2(\Omega)$, let $(u_h,\bs_h)\in V_h\times \bV_h$ be the unique solution of the mixed elliptic problem 
\begin{eqnarray}\label{w1-b}
- (\nabla\cdot \bs_h, \chi_h) &=& (f,\chi_h) \;\;\;\forall  \chi_h \in V_h,\\
\label{w2}
(\bs_h, \bv_h) +  (u_h,\nabla\cdot \bv_h) &=& 0 \;\;\; \forall \bv_h \in \bV_h.
\end{eqnarray}
Then, we define a pair of operators $(T_h,S_h):L^2(\Omega)\to V_h\times \bV_h$ as $T_hf=u_h$ and 
$S_hf=\bs_h$. With $T:L^2(\Omega) \to \mathcal{D}(\cL)$ being the inverse of the operator $\cL$, the following result holds  \cite[Lemma 1.5]{JT-1981}.
\begin{lemma}\label{lem:TT}
The operator $T_h:L^2(\Omega)\to V_h$  defined by $T_hf=u_h$ is selfadjoint, positive semidefinite on $L^2(\Omega)$ and
positive definite on $V_h$. Further, we have
$$
\|T_hf-Tf\|\leq ch^{1+\ell}\|f\|,\quad \ell=0,1.
$$
\end{lemma}

\subsection{Semidiscrete scheme}\label{sec:M1}

The semidiscrete mixed FE scheme is to seek a pair $(u_h,\bs_h):(0,T_0]\to  V_h\times \bV_h$ such that
\begin{eqnarray}\label{w1a-m}
(\cop u_h, \chi_h)- (\nabla\cdot \bs_h, \chi_h) &=& (f(u_h), \chi_h) \;\;\;\forall \chi_h \in V_h,\\
\label{w2a-m}
(\bs_h, \bv_h) +  (u_h,\nabla\cdot \bv_h) &=& 0 \;\;\;\forall \bv_h \in \bV_h,
\end{eqnarray}
with $u_h(0)=P_hu_0$.
Since $V_h$ and $\bV_h$ are finite-dimensional,  we can eliminate $\bs_h$ in the discrete level  
using \eqref{w2a-m} by writing it in terms of $u_h$. Therefore, substituting in \eqref{w1a-m}, we obtain  a system of time-fractional ODEs.
Existence and uniqueness can be shown using standard results from fractional ODE theory \cite{KST-2006}.

For the error analysis, define $e_u=u_h-u$ and $e_\bs =\bs_h-\bs$. Then, from (\ref{w1-m})-(\ref{w2-m}) and (\ref{w1a-m})-(\ref{w2a-m}), 
$e_u$ and $e_\bs$ satisfy 
\begin{eqnarray}\label{ee1}
(\cop e_u, \chi_h)- (\nabla\cdot e_\bs, \chi_h) &=& (f(u_h)-f(u),\chi) \;\;\;\forall \chi_h \in V_h,\\
\label{ee2}
( e_\bs, \bv_h) +  (e_u,\nabla\cdot \bv_h) &=& 0 \;\;\;\forall \bv_h \in \bV_h.
\end{eqnarray}
Using the projections $(\tilde{u}_h,\tilde{\bs})$, we split the errors
$e_u=(u_h- \tilde{u}_h)-(\tilde{u}_h- u)=:\theta-\rho,$ and 
$e_\bs=(\bs_h- \tilde{\bs}_h)-( \tilde{\bs}_h-\bs)=:\bx-\bz.$
From (\ref{ee1})-(\ref{ee2}), we then obtain 
\begin{eqnarray}\label{aa}
(\cop e_u, \chi_h)- (\nabla\cdot \bx, \chi_h) &=& (f(u_h)-f(u),\chi) \;\;\;\forall \chi_h \in V_h,\\
\label{bb}
(\bx, \bv_h) +  (\theta,\nabla\cdot \bv_h ) &=& 0 \;\;\;\forall  \bv_h \in \bV_h.
\end{eqnarray}
Next we state our results for the semidiscrete problem.

\begin{theorem} \label{thm:mixed-sm}
 Let  $(u,\bs)$ and $(u_h,\bs_h)$ be the solutions of \eqref{w1-m}-\eqref{w2-m} and 
 \eqref{w1a-m}-\eqref{w2a-m}, respectively, with $u_h(0)=P_hu_0$. 
 Then, for $u_0 \in \mathcal{D}(\cL^\nu)$,  $\nu\in  [0,1]$,  the following error estimates hold for $t>0$: 
\begin{equation}\label{es1}
 \|u_h(t)-u(t)\| \leq  c h^{1+\ell} \eh t^{-\alpha(1-\nu)},\quad \ell=0,1,
\end{equation}
and
\begin{equation}\label{es2}
 \|\bs_h(t)-\bs(t)\|\leq  c h \eh t^{-\alpha(1-\nu)},\quad \ell=1.
\end{equation}

\end{theorem}
%
%
\begin{proof} 
From the definition of the operator $T_h$ above, the  semidiscrete problem may also be written as
\begin{equation}\label{T_h-i3}
T_h\cop u_h+u_h=T_hf(u_h), \quad t>0,\quad u_h(0)=P_hu_0.
\end{equation}
Note that $T_h$ satisfies the properties in Lemma \ref{lem:TT} and $T_hP_h=T_h$.
Recalling the definition of the continuous operator $T$, see \eqref{pr-1},
the estimate \eqref{es1}
follows immediately from Theorem \ref{thm:semi}.
Now,   using \eqref{bb} and the standard inverse inequality: $\|\nabla\cdot\bx\|\leq ch^{-1}\|\bx\|$ $\forall \bx\in \bV_h$, we have 
$$
 \|\bx\|^2\leq \|\theta\|\,\|\nabla\cdot\bx\|\leq ch^{-1}\|\theta\|\,\|\bx\|.
$$
Further,  $\|\theta\|\leq \|e_u\|+\|\rho\|\leq c h^2 \eh t^{-\alpha(1-\nu)}$ by \eqref{est-1} and \eqref{es1}, so that
$
\|\bx(t)\|\leq c h \eh  t^{-\alpha(1-\nu)}.
$
Together with 
$\|\bz(t)\|\leq c h  \|\cL u(t)\|\leq c h t^{-\alpha(1-\nu)}\|u_0\|,$ this shows \eqref{es2}.
\end{proof}

\begin{remark} 
Avoiding the inverse inequality in the estimation of the flux variable $\bs$ seems to be 
challenging as
estimates for the first and second time derivatives of  $u$ are not available, even in the $H^{1}(\Omega)$-norm.
\end{remark}

\begin{remark} The following estimate holds in the  stronger $L^\infty(\Omega)$-norm:
\begin{equation}\label{es3}
\|u_h(t)-u(t)\|_{L^\infty(\Omega)}\leq c h |\ln h|t^{-\alpha(1-\nu)},\quad t>0, \quad \ell=1.
\end{equation}
Indeed, from \eqref{bb}, we may get $\|\theta(t)\|_{L^\infty(\Omega)}\leq C|\ln h|\,\|\bx(t)\|$, see \cite[Lemma 1.2]{JT-1981}. Noting also that $\|\bz(t)\|\leq c h|\ln h|  \|\cL u(t)\|$, see \cite[Theorem 1.1]{JT-1981},
 we obtain \eqref{es3}.
\end{remark}

\begin{remark} 
For the linear problem with  $f=0$ and  $u_0 \in \mathcal{D}(\cL^\nu)$,  $\nu\in  [1/2,1]$, the solution  $u(t)\in H^3(\Omega)$ for $t>0$. As a consequence, when $\ell=1$,  approximations of order $O(h^2)$ are achieved  for  both variables $u$ and $\bs$,   see \cite[Thorem 5.4]{Karaa-2018}. 
\end{remark}




\subsection{Fully discrete scheme}\label{sec:M3}

The  fully mixed FE scheme  based on the backward Euler CQ method is to find a pair $(U_h^n,\Sigma_h^n)\in V_h\times \bV_h$ such that for $n\geq 1$,
\begin{eqnarray}\label{w1a-BE}
(\dop U_h^n, \chi_h)- (\nabla\cdot \Sigma_h^n, \chi_h) &=& (\dop U_h^0,\chi_h)+ (f(U_h^n),\chi_h)\;\;\;\forall \chi_h \in V_h,\\
\label{w2a-BE}
(\Sigma_h^n, \bv_h) +  (U_h^n,\nabla\cdot \bv_h) &=& 0 \;\;\;\forall \bv_h \in \bV_h,
\end{eqnarray}
with $U_h^0=P_hu_0$. 
We now prove the following result.
\begin{theorem} \label{thm:mixed-sm-f-cc} 
 Let   $(u,\bs)$ be the solution \eqref{w1-m}-\eqref{w2-m}.
 Let  $(U_h^n,\Sigma_h^n)$ be the solution of \eqref{w1a-BE}-\eqref{w2a-BE} with $U_h^0=P_hu_0$. 
 Then, for $u_0 \in \mathcal{D}(\cL^\nu)$,  $\nu\in  (0,1]$,  the following error estimate holds for $t_n>0$:
\begin{equation}\label{es1-f-cc}
 \|U_h^n-u(t_n)\|\leq  c (\tau t_n^{\alpha\nu-1} + h^{1+\ell}t_n^{-\alpha(1-\nu)}),\quad \ell=0,1,
\end{equation}
where $c$ is independent of $h$ and $\tau$.
\end{theorem}
%
%
\begin{proof} In view of Lemma \ref{lem:TT}, the fully semidiscrete problem is rewritten as
\begin{equation}\label{T_h-in-cc}
T_h\dop U_h^n+U_h^n= T_h\dop U_h^0+T_hf(U_h^n), \quad n\geq 1,\quad U_h^0=P_hu_0.
\end{equation}
Recalling again  the definition of the continuous operator $T$,  the estimate \eqref{es1-f-cc}  follows immediately from Theorem \ref{thm:fully-1}. 
\end{proof}

We conclude this section by showing error estimates for the linear problem with $f=0$. 
The results are intended to complete the semidiscrete mixed FE error analysis  presented in \cite{Karaa-2018}.


\begin{theorem} \label{thm:mixed-sm-f}
 Let   $(u,\bs)$ be the solution \eqref{w1-m}-\eqref{w2-m} with $f=0$.
 Let  $(U_h^n,\Sigma_h^n)$ be the solution of \eqref{w1a-BE}-\eqref{w2a-BE} with $f=0$ and $U_h^0=P_hu_0$. 
 Then, for $u_0 \in \mathcal{D}(\cL^\nu)$,  $\nu\in  [0,1]$,  the following error estimate holds for $t_n>0$:
\begin{equation}\label{es1-f}
 \|U_h^n-u(t_n)\|\leq
 c (\tau t_n^{\alpha\nu-1} + h^{1+\ell} t_n^{-\alpha(1-\nu)})\|\cL^\nu u_0\|,\quad \ell=0,1.
\end{equation}
Furthermore, in the case $\nu=0$, 
\begin{equation}\label{es1-fs}
 \|\Sigma_h^n-\bs(t_n)\|\leq c (\tau t_n^{-\alpha/2-1}+ht_n^{-\alpha})  \|u_0\|,\quad \ell=1,
\end{equation}
where $c$ is independent of $h$ and $\tau$.
\end{theorem}
%
%
\begin{proof} 
The first estimate is given in Theorem \ref{thm:BE}. 
To derive \eqref{es1-fs}, we use \eqref{w1a-m}-\eqref{w2a-m} and \eqref{w1a-BE}-\eqref{w2a-BE} so that, with $u_h^n=u_h(t_n)$ and $\bs_h^n=\bs_h(t_n)$, we have
\begin{eqnarray*}
(\dop (U_h^n-u_h^n), \chi_h)- (\nabla\cdot (\Sigma_h^n-\bs_h^n), \chi_h) &=& -((\dop-\op) (u_h^n-u_0),\chi_h) \;\;\;\forall \chi_h \in V_h,\\
(\Sigma_h^n-\bs_h^n, \bv_h) +  (U_h^n-u_h^n,\nabla\cdot \bv_h) &=& 0 \;\;\;\forall \bv_h \in \bV_h.
\end{eqnarray*}
Choose $\chi_h=U_h^n-u_h^n$ and $\bv_h=\Sigma_h^n-\bs_h^n$ so that
\begin{equation}\label{ms-1}
 \|\Sigma_h^n-\bs_h^n\|^2+ (\dop (U^n_h-u_h^n),U^n_h-u_h^n)= -((\dop-\op)(u_h^n-u_0),U^n_h-u_h^n).
\end{equation}
From  \eqref{T_h-i3} and \eqref{T_h-in-cc}, we see that $U_h^n-u_h^n=G(\partial_\tau)u_0-G(\partial_t)u_0$, where  
$G(z)=(z^{-\alpha}I+T_h)^{-1}T_h$. Hence,
\begin{eqnarray*}
\dop(U^n_h-u_h^n)&=& (\dop U^n_h - \op u^n_h) - (\dop-\op) u^n_h \\
                 &=& (\tilde G(\partial_\tau)- \tilde G(\partial_t))u_0 - (\dop-\op) u^n_h,
\end{eqnarray*}
where $\widetilde G(z)=z^\alpha G(z)$. Inserting this result in \eqref{ms-1}, we get
\begin{equation}\label{ms-1a}
 \|\Sigma_h^n-\bs_h^n\|^2+ ((\widetilde G(\partial_\tau)- \widetilde G(\partial_t))u_0,U^n_h-u_h^n)= ((\dop-\op)u_0,U^n_h-u_h^n),
\end{equation}
and therefore
$$
\|\Sigma_h^n-\bs_h^n\|^2 \leq \left(\|(\widetilde G(\partial_\tau)- \widetilde G(\partial_t))u_0 \|+\|(\dop-\op)u_0\|\right)
\|U^n_h-u_h^n \|.
$$
Applying Lemma \ref{lem:Lubich} (with $\mu=-\alpha$, $\sigma=1$), we obtain
$\|(\dop-\op)u_0\| \leq c\tau t_n^{-\alpha-1}\|u_0\|$. Similarly, $\|\tilde G(z)\|\leq c|z|^{\alpha}\, \forall z\in \Sigma_\theta$ implies that $\|(\tilde G(\dop)- \tilde G(\op))u_0\|\leq c\tau t_n^{-\alpha-1}\|u_0\|$. The estimate \eqref{es1-fs} follows now since $\|\bs_h(t_n)-\bs(t_n)\|\leq cht^{-\alpha}\|u_0\|$, see \cite{Karaa-2018}. This completes the proof.
\end{proof}

\begin{remark}
For  problem \eqref{a} with  $\cL u = -\mbox{div} [A(x)\nabla u]+\kappa(x)u$, one may consider an expanded mixed FE method  by setting
${\bf q}=\nabla u$ and $\bs = A{\bf q}$, see \cite{Chen-1998}. Thus, the  scalar unknown $u$, its gradient and its flux $\bs$ are treated explicitly. The new mixed Galerkin method is to find   $(u,{\bf q},\bs):(0,T_0]\to  V\times \bV\times \bV$ such that
\begin{eqnarray*}\label{w1-m-ne}
(\cop u, \chi)- (\nabla\cdot \bs, \chi)+(\kappa u, \chi) &=& (f(u),\chi) \;\;\;\forall \chi \in V,\\
\label{w2-m-ne}
({\bf q}, \bv) +  (u,\nabla\cdot \bv) &=& 0\ \;\;\;\forall \bv \in \bV, \\
\label{w3-m-ne}
(\bs, \bv) - (A {\bf q},\bv) &=& 0\ \;\;\;\forall \bv \in \bV, 
\end{eqnarray*}
with $u(0)=u_0$. We define the semidiscrete approximation as the triple $(u_h,{\bf q}_h,\bs_h):(0,T_0]\to  V_h\times \bV_h\times \bV_h$  satisfying
\begin{eqnarray*}\label{w1-m-ne-p}
(\cop u_h, \chi_h)- (\nabla\cdot \bs_h, \chi_h)+(\kappa u_h, \chi_h) &=& (f(u_h),\chi_h) \;\;\;\forall \chi_h \in V_h,\\
\label{w2-m-ne0-p}
({\bf q}_h, \bv_h) +  (u_h,\nabla\cdot \bv_h) &=& 0\ \;\;\;\forall \bv_h \in \bV_h, \\
\label{w3-m-ne-p}
(\bs_h, \bv_h) - (A {\bf q}_h,\bv_h) &=& 0\ \;\;\;\forall \bv_h \in \bV_h, 
\end{eqnarray*}
with $u_h(0)=P_hu_0$. Then, in order to extend  the previous analysis to the  expanded mixed FE method, one 
may prove a result similar to Lemma \ref{lem:TT}. This can be achieved by using the approximation properties derived in \cite[Theorem 4.8]{Chen-1998}.

\end{remark}

\section {Numerical Experiments} \label{sec:numer}
In this part, we present  numerical examples to verify the theoretical results. We consider problem \eqref{a} and its mixed form  \eqref{w1-m}-\eqref{w2-m}  in the unit square $\Omega=(0,1)^2$ with
$\cL=-\Delta$. We choose $f(u)=\sqrt{1+u^2}$ and perform numerical tests with
the following smooth and nonsmooth initial data: 

\begin{itemize}
\item [(a)] $u_0(x,y)=xy(1-x)(1-y)\in \dot H^2(\Omega)$.
\item [(b)] $u_0(x,y)=\chi_D(x,y)\in \dot H^{1/2-\epsilon}(\Omega)$, $\epsilon>0$, 
\end{itemize}
where $\chi_D(x,y)$ denotes the characteristic function of the domain  $D:=\{(x,y)\in\Omega:x^2+y^2\leq 1\}$.

Since the error analysis of the standard conforming Galerkin FEM has thoroughly been  investigated, see for instance \cite{JLZ2016}, we shall mainly focus on spatial errors from  nonconforming and mixed FEMs. The backward Euler CQ method is used for the time discretization, 

In the computation, we divide the domain $\Omega$ into regular right triangles with $M$ equal subintervals of length $h=1/M$ on each side of the domain. We choose $\alpha = 0.5$ and the final time $T=0.1$. Since the exact solution is difficult to obtain, we compute  a reference solution on a refined mesh in each case. All the numerical results are obtained using FreeFEM++ \cite{FreeFem++}.

To check the spatial discretization errors, we display in  Table \ref{table:1} the $L^2(\Omega)$-norm of the errors  in the discrete solutions in cases (a) and (b), computed by the Crouzeix-Raviart nonconforming  FEM $(P1nc)$ based on the numerical scheme \eqref{H1} and its analogue used with the standard conforming FEM. 
From the table, we observe a convergence rate of order  $O(h^2)$  for smooth and nonsmmoth initial data,
which confirms the theoretical convergence rates. Similar results have been obtained with different values 
of $\alpha$. We notice that the nonconforming method yields slightly better results.


\begin{table}[h]
{\footnotesize
\begin{center}
\caption{$L^2$-errors for cases (a) and (b); $P1$ and $P1nc$ FEMs.}
\label{table:1}
\begin{tabular}{|r|cc | cc| cc | cc|}
\hline
 & \multicolumn{4}{|c|}{$P1$ (Conforming)}& \multicolumn{4}{|c|}{$P1nc$ (Crouzeix-Raviart)} 
\\
\hline 
 & \multicolumn{2}{|c|}{Problem (a)} & \multicolumn{2}{|c|}{Problem (b)} &
 \multicolumn{2}{|c|}{Problem (a)} & \multicolumn{2}{|c|}{Problem (b)}
\\
\hline
$M$ & $L^2$-error  & Rate& $L^2$-error  & Rate & $L^2$-error  & Rate& $L^2$-error  & Rate
 \\
\hline
8  & 1.83e-3 &      &  1.82e-3 &      &  9.26e-4  &        & 9.39e-4 &\\
16 & 4.47e-4 & 1.97 &  4.64e-4 & 1.98 &  2.43e-4  &  1.93  & 2.46e-4 & 1.93\\
32 & 1.15e-4 & 2.02 &  1.14e-4 & 2.01 &  6.15e-5  &  1.98  & 6.23e-5 & 1.99\\
64 & 2.73e-5 & 2.06 &  2.71e-5 & 2.05 &  1.52e-5  &  2.01  & 1.54e-5 & 2.02\\
\hline
\end{tabular}
\end{center}
}
\end{table}


For the mixed FEM, we perform the computation using the lowest-order Raviart-Thomas FE spaces
$(RT0,P0)$ and  $(RT1,P1dc)$, where $P0$ and $P1dc$ denote the sets of piecewise constant and linear functions, respectively. Here, we adopt the notation used in FreeFEM++. In our tests, we include  the case $\alpha=1$ (i.e., the parabolic case) in order to investigate the effect of the solution regularity.

\begin{table}[t]
{\footnotesize
\begin{center}
\caption{$L^2$-errors and convergence rates for case (a); mixed FEMs.}
\label{table:2}
\begin{tabular}{|r|cccc|}
\hline
$M$&
$\|u-u_h\|_{L^2}$ & Rate& $\|{\boldsymbol\sigma}-{\boldsymbol\sigma}_h\|_{L^2}$ & Rate
\\
\hline
\multicolumn{5}{|c|}{$(RT1,P1dc),\; \alpha=0.5$}\\
\hline
 8 & 4.63e-4  &      & 3.31e-3 &  \\
16 & 1.11e-4  & 2.05 & 8.23e-4 & 1.86 \\
32 & 2.49e-5  & 2.16 & 2.49e-4 & 1.89 \\
64 & 6.28e-6  & 1.99 & 5.48e-5 & 1.96 \\
\hline
\multicolumn{5}{|c|}{$(RT1,P1dc),\; \alpha=1$}\\
\hline
8  & 4.61e-4 &      & 3.32e-3 &  \\
16 & 1.10e-4 & 2.06 & 9.24e-4 & 1.86 \\
32 & 2.47e-5 & 2.15 & 2.49e-4 & 1.90 \\
64 & 6.24e-6 & 1.99 & 6.45e-5 & 1.96 \\
\hline
\multicolumn{5}{|c|}{$(RT0,P0),\; \alpha=0.5$} \\
\hline
 8 & 5.41e-3 &      & 2.73e-2 &  \\
16 & 2.71e-3 & 1.00 & 1.40e-2 & 0.97 \\
32 & 1.35e-3 & 1.00 & 7.04e-3 & 0.97 \\
64 & 6.81e-4 & 1.00 & 3.52e-3 & 1.00 \\
\hline
\end{tabular}
\end{center}
}
\end{table}

\begin{table}[ht]
{\footnotesize
\begin{center}
\caption{$L^2$-errors and convergence rates for case (b); mixed FEMs.}
\label{table:3}
\begin{tabular}{|r|cccc|}
\hline
$M$&
$\|u-u_h\|_{L^2}$ & Rate& $\|{\boldsymbol\sigma}-{\boldsymbol\sigma}_h\|_{L^2}$ & Rate
\\
\hline
\multicolumn{5}{|c|}{$(RT1,P1dc),\; \alpha=0.5$}\\
\hline
 8 & 1.21e-3 &      & 8.83e-3 &  \\
16 & 2.93e-4 & 2.04 & 2.69e-3 & 1.71 \\
32 & 6.72e-5 & 2.12 & 8.36e-4 & 1.70 \\
64 & 1.62e-5 & 2.05 & 2.54e-5 & 1.71 \\
\hline
\multicolumn{5}{|c|}{$(RT1,P1dc),\; \alpha=1$}\\
\hline
8  & 1.35e-3 &      & 5.40e-3 &  \\
16 & 3.21e-4 & 2.07 & 1.39e-3 & 1.97 \\
32 & 7.27e-5 & 2.14 & 3.59e-4 & 1.97 \\
64 & 1.57e-5 & 2.19 & 9.39e-5 & 1.94\\
\hline
\multicolumn{5}{|c|}{$(RT0,P0),\; \alpha=0.5$} \\
\hline
8  & 1.37e-2 &      & 6.60e-2 &  \\
16 & 6.92e-3 & 0.99 & 3.38e-2 & 0.97 \\
32 & 3.46e-3 & 1.00 & 1.07e-2 & 0.99 \\
64 & 1.74e-3 & 1.00 & 8.52e-3 & 1.00 \\
\hline
\end{tabular}
\end{center}
}
\end{table}


The numerical results for problem (a) are given in  Table \ref{table:2}. They show a convergence rate of order $O(h)$  in the case of $(RT0,P0)$ and   of order  $O(h^2)$ in the case of $(RT1,P1dc)$ for both values $\alpha=0.5$ and $\alpha=1$,  which agrees well with the theoretical results.

In  Table \ref{table:3}, we present the numerical results for problem (b). The results reveal that the convergence rates are maintained in the cases of $(RT0,P0)$ with $\alpha=0.5$ and $(RT1,P1dc)$ with $\alpha=1$, which agrees well with our convergence analysis. By contrast, the convergence rate reduces to $O(h^{1.7})$ in the case of $(RT1,P1dc)$ with $\alpha=0.5$. This confirms our prediction that the optimal 
$O(h^{2})$ convergence rate is no longer attainable when the initial data is not smooth. Note  that, since the numerical results do not show  a convergence rate of $O(h)$, this may be seen as an unexpected result.  
However, as the initial data $u_0$ has some smoothness, $u_0$ is roughly in $\dot H^{1/2-\epsilon}(\Omega)$ for some $\epsilon>0$, 
the numerical results do not contradict our theoretical findings. Indeed, the smoothness of the particular initial data $u_0$ could then have  a positive effect on the convergence rate. This fact has also been observed in the study of 
a homogeneous linear time-fractional problem with time-dependent coefficients \cite{KP-2019}.

\end{document}